\documentclass[a4paper,12pt]{article}

\usepackage{a4,amssymb,amsmath,amsthm,url,xcolor,enumerate,float,lineno}
\usepackage{bezier,amsfonts,amssymb,graphicx,amsthm,url}
\usepackage[utf8]{inputenc}
\usepackage{amsmath}
\usepackage{amsfonts}
\usepackage{amssymb,amsthm}
\usepackage{url}

\newtheorem{theorem}{Theorem}[section]

\newtheorem{corollary}[theorem]{Corollary}
\newtheorem{lemma}[theorem]{Lemma}
\newtheorem{theorem*}{Theorem}
\newtheorem*{theoremx}{Theorem}
\newtheorem*{theorem1}{Theorem 1}
\newtheorem*{theorem2}{Theorem 2}

\def\kaxxa{{\vcenter {\hrule height .2mm
\hbox{\vrule width .2mm height 2mm \kern 2mm
\vrule width .2mm} \hrule height .2mm}}}

\title{The feasibility problem for line graphs}
\author{Yair Caro\\ University of Haifa-Oranim \and Josef Lauri \\ University of Malta \and Christina Zarb \\ University of Malta}
\date{ }

\DeclareGraphicsExtensions{.pdf,.png,.jpg}

\begin{document}



\maketitle

\begin{abstract}
We consider the following feasibility problem: given an integer $n \geq 1$ and an integer $m$ such that $0 \leq m \leq \binom{n}{2}$, does there exist a line graph $L = L(G)$ with exactly $n$ vertices and $m$ edges ?

We say that a pair $(n,m)$ is non-feasible if there exists no line graph $L(G)$ on $n$ vertices and $m$ edges, otherwise we say $(n,m)$ is a feasible pair.  Our main result shows that for fixed $n\geq 5$, the values of $m$ for which $(n, m)$ is a non-feasible pair, form disjoint blocks of consecutive integers which we completely determine.  On the other hand we prove, among other things, that for the more general family of claw-free graphs (with no induced $K_{1,3}$-free subgraph), all $(n,m)$-pairs in the range $0 \leq  m \leq \binom{n}{2}$ are feasible pairs.

\end{abstract}



\section{Introduction}

We consider the feasibility problem for line graphs, with some extensions to other families to demonstrate the context which is an umbrella for many well-known graph theoretic  problems.  All graphs in this paper are simple graphs, containing no loops or multiple edges.   

\bigskip

\emph{The feasibility problem}:

 \bigskip

Given $F$ a family of graphs and a pair $(n,m)$, $n \geq 1$, $0 \leq m \leq \binom{n}{2}$,   the pair $(n,m)$  is called \emph{feasible} (for $F$) if there is a graph $G \in  F$, with $n$ vertices and $m$ edges.  Otherwise $(n,m)$ is called a \emph{non-feasible} pair.  A family of graphs $F$ is called feasible if for every $n \geq 1$, every pair $(n,m)$ with $0 \leq  m \leq \binom{n}{2}$  is feasible,  and otherwise $F$ is called non-feasible.   The problem is to determine whether $F$ is feasible or not, and to find all feasible pairs, respectively non-feasible pairs for $F$.

 \bigskip

\emph{The minimum/maximum  non-feasible pair problem }:

 \bigskip

For a non-feasible family $F$, determine for any $n$ the minimum/maximum value of $m$ such that the pair $(n,m)$  is non-feasible.

An immediate example is the family of all graphs having no copy of a fixed graph $H$  which is a major problem in extremal graph theory.  Here, if $ex(n,H)$ is the corresponding Tur{\'a}n number then $m = ex(n,H) +1$ is the smallest $m$ such that the pair $(n,m)$ is non-feasible, and of course these``H-free families" are all non-feasible families \cite{furedi2013history,turan1941external}. 
A simpler example is that of the family of all connected graphs.  Here we know that given $n$ and  $n-1 \leq m \leq \binom{n}{2}$, the pair $(n,m)$ is feasible, namely realised by a connected graph, however for $0 \leq  m \leq  n-2$, no pair $(n,m)$ is feasible and the maximum value of $m$ for a given $n$ which gives a non-feasible pair is $m=  n-2$.

However there are many more related problems already discussed in the literature of  which we mention for example \cite{axenovich2021absolutely,barrus2008graph,ERDOS199961}.

So the reader can think of many other problems that can be formulated as feasibility problems in the framework suggested above.

In Section 2 we collect some basic facts about feasible families and we further illustrate  this notion of feasibility by demonstrating, for example, that the  families of $K_{1,3}$-free graphs, chordal graphs, and paw-free graphs (a paw is the graph which consists of $K_3$ and an attached leaf) are all feasible families.

Yet our main motivation is to fully characterize the non-feasible pairs $(N,M)$ for the  family of all line graphs.   Here we stress that we reserve along the paper, the notation $( N,M)$ for a line graph $L(G)$  having $N$ vertices and $M$ edges  to make a clear distinction from the underlying graph $G$ having $n$ vertices and $m$ edges where $m = N$.  Observe that isolated vertices in the underlying graph have no role.  The family of all line graphs is easily seen to be non-feasible  since already  the pair $(N,M) =  (5,9)$ is realised only by the graph $K_5 \backslash \{e\}$ (where $e$ is an edge) which is not a line graph and belongs to the famous list of Beineke forbidden subgraphs characterizing line graphs \cite{beineke70,cvetkovic1974some,FAUDREE199787,SOLTES1994391,LAI200138}.

Some classical  simple results  will be frequently used --- we mention  the following here: 

 \begin{itemize}
\item[-]{Fact 1:  If $G$ is a graph on $n(G) = n$ vertices and $e(G) = m$ edges with degree sequence $d_1 \leq \ldots \leq d_n$, then  the line graph $L(G)$  has $n(L(G)) = m$ vertices and $e(L(G))  =  \sum_{j=1}^{n} \binom{d_j}{2}$. \cite{harary1969graph} }
 \item[-]{Fact 2:  A sequence of $n +1$  positive integers  $1 \leq d_1 \leq \ldots \leq d_{n+1}$  is the degree sequence of a tree if and only if  $\sum_{j=1}^{n+1}d_j= 2n$. \cite{bose2008characterization} } 
 \end{itemize}


Quite naturally, by Fact 1 and  in the context of line graphs,  the feasibility of a pair  $(N,M)$  is closely related  to the number-theoretic problem of representing non-negative integers by a sum of triangular numbers,  which dates back to Gauss \cite{ewell} who proved that every  non-negative integer $n$ is representable by the sum of at most three triangular numbers (hence by exactly three triangular numbers).  For further details see \cite{REZNICK1989199}.

Yet we stress that the main distinction between the number-theoretic problem  and the feasibility problem for line graphs of all graphs lies in the fact that in order to have a line graph  $L(G)$ with $N$ vertices and $M$ edges, making $( N,M)$ a feasible pair we require that $\sum_{j=1}^{s} d_ j  = 2N$  is obtained by a graphical degree sequence with a realizing/underlying graph  $G$ having $m=N$ edges (the number of vertices is not important) and the line graph $L(G)$ has $n(L(G)) =  N$  while $\sum_{j=1}^{s} \binom{d_j}{2} =  M$.   For example none of the representations of the integer 9  as a sum of triangular numbers $\sum_{j=1}^{s} \binom{x_ j}{2}  = 9$  belongs to any  graphical sequence realised  by a graph having exactly 5 edges, as this would imply that  $K_5 \backslash \{e\}$ (where $e$ is an edge) is a line graph, which it is not.


Section 3  is a sort of warmup to the main result, allowing us to exhibit the main tools and methods of the  proof of the main theorem of this paper.  In this section we consider  a lower bound and upper bound for the minimum non-feasible pair $(N,M)$ for the family of line graphs of all acyclic graphs, and show that there are positive constants  $0 < c_1 \leq c_2$  such that for given $N$, the minimum  $M$ which makes $(N,M)$ non-feasible satisfies $\frac{N^2}{2} -  c_2N\sqrt{N}  \leq M \leq   \frac{N^2}{2}  -  c_1N\sqrt{N}$   and these bounds can be compared with \cite{REZNICK1989199}.  So, already,  non-feasible pairs $(N,M)$  for the line graphs of acyclic graphs are possible only for  $M  \geq    \binom{N – c\sqrt{N}}{2}$.

Section 4 deals with the feasibility problem for the family of all line graphs, the main aim of this paper, and requires several preparatory lemmas before proving the following main theorems of this paper which we state here. (If $a<b$ are positive integers, then $[a,b]$ will denote the set of all integers $p$ such that $a \leq p \leq b$.)

 \begin{theorem1}[The Intervals Theorem] 
For $N  \geq 5$,  all the values of $M$ for which $(N, M)$ is a non-feasible pair for the family of all line graphs, are exactly  given by all integers $M$ belonging to the following intervals:
\[\left [ \binom{N-t}{2} +\binom{t+2}{2}, \ldots,  \binom{N-t+1}{2}-1 \right ] \mbox{ for }  1 \leq t < \frac{ -5 + \sqrt{8N +17}}{2}\]

Observe that if  $\frac{ -5 + \sqrt{8N +17}}{2}$ is not an integer then $t= \left \lfloor \frac{ -5 + \sqrt{8N +17}}{2} \right \rfloor$ while if $\frac{ -5 + \sqrt{8N +17}}{2}$ is an integer then $t=\frac{ -5 + \sqrt{8N +17}}{2}-1$.

\end{theorem1}


\begin{theorem2}[The minimum non-feasible pair]
For  $N \geq 2$,  the minimum  value of $M$ which makes $(N,M)$ a non-feasible pair, for the family of all line graphs,  is $\binom{N-t}{2} + \binom{t +2}{2}$  where :
\begin{enumerate}
\item{$t  = \left \lfloor \frac{ -5 + \sqrt{8N +17}}{2} \right \rfloor$    if  $\frac{ -5 + \sqrt{8N +17}}{2}$  is not an integer.}
\item{$t =   \frac{ -5 + \sqrt{8N +17}}{2}  - 1$ if  $\frac{ -5 + \sqrt{8N +17}}{2}$  is an integer.}
\end{enumerate}
\end{theorem2}

The following intervals, demonstrating Theorem 1 with $N = 27$, give all  the values of $M$ for which the pairs $ (27,M)$  are non-feasible for all line graphs: $ t=5 \mbox{ } [252]  ,  t=4 \mbox{ }[267-275]  , t=3 \mbox{ }[ 286-299 ] , t=2\mbox{ } [ 306-324] , \mbox{ } t=1\mbox{ } [328-351]$.

The following table, demonstrating Theorem 2, gives for  $N \leq 30$, the values of  $k$  and the smallest value of $M$ for which $(N,M)$  is a non-feasible pair for the family of all line graphs.
\bigskip

\begin{minipage}[c]{0.33\textwidth}
\begin{tabular}{|c|c|}
\hline

N  & M\\
\hline

1  &*\\2  &*\\3  &*\\4 &* \\ 5&9 \\6&13 \\7&18 \\8&24\\9&27\\10&34\\
\hline
\end{tabular}
\end{minipage}
\begin{minipage}[c]{0.33\textwidth}
\begin{tabular}{|c|c|}\hline

N &M\\
\hline
11&42\\12&51\\13&61\\14&65\\15&76\\16&88\\17&101\\18&115\\19&130\\20&135\\
\hline
\end{tabular}
\end{minipage}
\begin{minipage}[c]{0.33\textwidth}
\begin{tabular}{|c|c|}\hline

N &M\\
\hline
21&151\\22&168\\23&186\\24&205\\25&225\\26&246\\27&252\\28&274\\29&297\\30&321\\
 \hline
\end{tabular}
\end{minipage}

\bigskip

It is worth noting that  our proof has a part in which we used a computer program to check both theorems up to  $N = 35$ since the computations and estimates we used in the  proof   apply for $N \geq 33$ (though with further efforts it can be reduced to $N \geq 24$  which we decided to avoid).


We shall mostly follow the standard graph-theoretic notations and definitions as used in \cite{west2001}.  Recall $n(G)$, $e(G)$, $\Delta(G)$  and $\delta(G)$  are the number of vertices, number of edges, maximum degree and minimum degree of $G$ respectively.  Whenever there is  distinction we shall make the notation/definition clear in the body of the paper when the term or definition first appears.

\section{Basic Feasible Families}

We say that a graph $G$ is induced $H$-free or just $H$-free (when there is no ambiguity) if $G$ has no induced copy of $H$. 


\medskip

\noindent \textbf{The universal elimination procedure (UEP)}

\medskip

We start with $K_n$  and order the vertices $v_1,\ldots, v_n$.  We now delete at each step an edge incident with $v_1$ until $v_1$ is isolated.  We then repeat the process of step by step deletion of the edges incident with $v_2$, and continue until we reach  the empty graph on $n$ vertices.

Along the process, for any pair $( n,m)$, $ 0  \leq m \leq \binom{n}{2}$, we have  a graph $G$ with $n$ vertices and $m$ edges.
\begin{lemma}

The maximal induced subgraphs of $K_n$ obtained when applying  UEP on $K_n$  are of the form $H(p,q,r) =  ( K_p  \backslash S_q )  \cup rK_1$, where $S_q$ is the star on $q$ edges, $p \geq q \geq 0$  and $p +r = n$.   
\end{lemma}
 
\begin{proof}
This is immediate from the definition and description of UEP.
\end{proof}

Already the  UEP supplies many feasible families as summarized in the following corollary.

\begin{corollary}
The following families of graphs obtained by  applying the  UEP  are feasible:
\begin{enumerate}
\item{ induced  $K_{1,r}$-free for $r\geq 3$, where $K_{1,r}$ is the star with $r$ leaves.}
\item{  induced  $P_r$-free  for $r \geq 3$, where $P_r$ is the path on $r$ edges.}
\item{ induced  $rK_2$-free for $r \geq 2$ where $rK_2$ is the union of $r$ disjoint edges.} 
\item{chordal (reference to chordal graphs \cite{west2001}).}
\end{enumerate}   
\end{corollary}
\begin{proof}
This is immediate from the definition and description of UEP.
\end{proof}

 


We now show that the family of all induced paw-free graphs is a feasible family, depite the paw graph itself being $H(4,2,0 )$ and hence cannot be proven using UEP.

\begin{theorem}
The family of all induced paw-free graphs is feasible.
\end{theorem}
\begin{proof}
We proceed  by induction on the number of vertices $n$.  It is trivial for $n = 1 , 2 , 3$.  So assume it is proven up to $n-1$ and we prove it for $n$. 

We will show first that all the pairs $( n,m)$  where  $\binom{n -1}{2}+1 \leq m \leq  \binom{n}{2}$ are feasible.  Let the vertices of the graph be $v_1, \ldots v_n$.  Then we can delete all edges incident with $v_n$, and apply induction on the range $m \leq  \binom{n -1}{2}$ and the extra isolated vertex $v_n$.

So we will show that we can delete $t$ edges from $K_n$, $0 \leq t \leq n -2$, without having an induced paw. This will cover the required range between $\binom{n-1}{2}+1$ and $\binom{n}{2}$.  Let  $n = 3k +r$  where $0\leq r < 3$.  We claim that for every  $t =  3q + p$,   $0\leq t \leq n-2$, we can delete induced  $qK_3 \cup pK_2$ from $K_n$.  Indeed consider the following cases:
\begin{enumerate}
\item{if $r  = 1$  then we only have to check that we can delete $3k - 1$ edges.  Take $(k-1)K_3$ and we still have 4 vertices untouched. So we can take  a further $2K_2$, altogether $3(k-1)+2 = 3k -1  = n-2$.  }
 \item{if $r  = 2$ then we only have to check that  we can delete $3k$ edges.  Take $kK_3$  altogether $3k = n-2$ edges.}

\item{if $r = 0$  then we only have to check that we can delete $3k -2  = 3(k-1) +1$ edges.  Take $(k-1)K_3$ and we still have 3 vertices  untouched so we can remove $K_2$ giving altogether $3(k-1) +1 = 3k - 2 = n-2$.}
\end{enumerate}

Consider the graphs obtained by removing the edges.  Suppose there is an induced paw.  Let $z$ be the vertex of degree 1 in the induced paw, which is adjacent to a vertex $w$ of degree 3 in the induced paw, and $w$ is adjacent to two other vertices $x$ and $y$ of degree 2 in the induced paw, with $z$ not adjacent to either $x$ or $y$.  Then in the graphs above, obtained by removing $qK_3 \cup pK_2$, the vertex $z$ must be connected to some $K_3$ via exactly one edge, missing exactly two edges (incident to $z$ in the complete graph) to this $K_3$.  These two missing edges make up an induced $K_{1,2}$.  However this is impossible since the missing edges form $qK_3 \cup pK_2$ by construction.

\end{proof}


\section{Non-feasibility of the family of all line graphs of acyclic graphs}

We start this section with a few words about convexity arguments which we shall use in this section and in section 4.
We first consider the well-known Jensen inequality\cite{jensen1906fonctions}.
\begin{theoremx}[Jensen]
If $f$ is a real continuous function that is convex, then
\[f \left (\frac{\sum_{i=1}^n x_i}{n} \right ) \leq \frac{\sum_{i=1}^n f(x_i)}{n}.\]   Equality holds if and only if   $ x_1=x_2=\ldots=x_n$  or  if $f$ is a linear function on a domain containing $x_1, x_2,\ldots,x_n$.
\end{theoremx}

Now, since $\binom{x}{2}$  is a convex, stricly monotome (for $x \geq 1$) function we can apply Jensen's inequality together with Fact 1 and Fact 2 (given in the introduction) to gain information  about  graphs with a given number of edges versus the number of edges in their line graphs .   

In particular, the following simple facts, referred to in this paper in short as ``by convexity", are used many times (for similar applications see \cite{caro2022index}): 
\begin{enumerate}
\item{For $1 \leq x \leq y$,  $\binom{x}{2}+\binom{y}{2} < \binom{x-1}{2} + \binom{y +1}{2}$.  For example, take vertices $x$, $y$ and $z$ in a graph $G$, where $deg(y)  \geq deg(x)$, $x$ and $z$ are adjacent while $z$ and $y$ are non-adjacent. We replace the edge $xz$ with the edge $zy$  to get a graph $G^*$ such that $e(G) = e(G^*)$  but e$(L(G^*)) > e(L(G))$.}
\item{For $1 \leq x \leq y$, $\binom{x}{2}+\binom{y}{2} < \binom{x+y}{2}$.  For example, take non-adjacent vertices $u$  and $v$  in $G$ such that $N(u)$ and $N(v)$ have no vertex in common.  We identify $u$ and $v$, namely by replacing them by a vertex $w$ adjacent to all $N(u) \cup N(v)$  to get a graph $G^*$, with $e(G) = e(G^*)$  but $e(L(G^*)) > e(L(G))$.  }
\end{enumerate}

When  more involved applications of convexity are used, we shall give the full details. 

A \emph{star-forest} $F$ is a forest whose components are stars (not necessarily of equal order), that is $F =  \bigcup_{j=1}^p K_{1,n_j}$.

Observe that if $F$ is a star forest then  $n(F) = \sum_{j=1}^p (n_j +1)$ while $e(F) = \sum_{j=1}^p n_j  = n(L(H))$ and $e(L(F))=\sum_{j=1}^p \binom{n_j}{2}$.

A crucial role in this section is played by the function
 \[g(N, \Delta)  =  \max \{  e(L(F)): \mbox{ $F$ is acyclic, $e(F)=N$, $\Delta(F)=\Delta$ and $\delta(G) \geq 1$} \}.\]  

Recall that isolated vertices in $G$  are not represented in $L(G)$ and have no impact on $e(G)=n(L(G))$, $\Delta(G)$ and $e(L(G))$.
   

Since in section 3 we deal with underlying acyclic graphs and their corresponding line graphs,  we say that (the line graph of  the) acyclic graph $F$ realises $g(N,\Delta)$  if $e(L(F))  =  g(N,\Delta)$.

Lastly we  frequently use $\Delta=\Delta(G)$ as the maximum degree of the underlying graph, for smooth presentation of formulas and computations.

\begin{lemma} \label{tree}
For $N \geq \Delta \geq 2$, $g(N, \Delta)$ is realised only by line graphs of  trees having $\Delta(T)=\Delta$.
\end{lemma}
\begin{proof}
Assume on the contrary that $g(N,\Delta)$ is realised by a (line graphs of a ) forest $F$, having $e(F)=N$ and $\Delta(F)=\Delta$, with at least two components (and no isolated vertices).  Then we can identify two leaves from distinct components of the forest $F$ and replace them by a vertex that is adjacent to their neighbours to get an acyclic graph $H$ with less components.  Clearly $\Delta(F)=\Delta(H)$,  (since $\Delta \geq 2$), but we have $e(L(H)) > e(L(F))$ by Fact 1 and convexity, a contradiction to the claim that $L(F)$ realises $g(N,\Delta)$.  


So the value of $g(N,\Delta)$  is indeed only realised by line graphs of trees on $N$ edges with maximum degree $\Delta$.
\end{proof}

\begin{lemma} \label{lemma3.2}

The following facts about $g(N,\Delta )$ hold:
\begin{enumerate}
\item{$g(N,1)  = 0  .$}
\item{Suppose $\Delta(F)=\Delta \geq 2$ and   $N-1 \equiv  k \pmod{\Delta – 1}$,  for  some $k$, $0 \leq k \leq \Delta -2$. Then  \[g(N,\Delta ) =  \frac{(N-k-1 )\Delta}{2} + \binom{k+1}{2} \leq \frac{(N-1)\Delta}{2}.\]}
\item{$g(N+1,\Delta ) \geq g(N,\Delta)+1$  for $\Delta \geq 2$.}
\item{$g(N,\Delta +1) \geq g(N,\Delta) +1$  for $N > \Delta$. }
\end{enumerate}
\end{lemma}
\begin{proof} 
$\mbox{ }$\\
\begin{enumerate}
\item{ This is obvious as the line graph of $NK_2$  has no edges.}
\item{By Lemma \ref{tree}, $g(N,\Delta)$ is realised only by line graphs of trees.  For $\Delta = 2$,  the only tree is $P_{N+1}$ and the result is immediate since $k = 0$ and $\Delta=2$.

So we assume $\Delta \geq 3$.  Consider any tree $T$ whose line graph $L(T)$  realises $g(N,\Delta)$.  We claim that $T$ can have at most one vertex of degree $r$, where $2 \leq r  \leq \Delta -1$.

Indeed, consider the degree sequence of $T$, $1 = d_1 \leq d_2  \leq \ldots \leq d_N \leq d_{N+1} =  \Delta$, where $\sum_{j=1}^{N+1} d_j=2N$.

If there are two degrees $2 \leq d_j \leq d_{j +1} < \Delta$,  we  take $d^*_j = d_j – 1$  and $d^*_{j+1} = d_{j+1} +1$ and rearrange the sequence, which is, again by Fact 2, realised by a tree $T^*$ and by convexity $e(L(T^*))>e(L(T))$, a contradiction to the claim that $L(T)$ realises $g(N,\Delta)$.  So we can continue this ``switching" until at most one vertex of degree  $r$, $2 \leq r \leq \Delta -1$, is left.     

We conclude  that a tree whose line graph realises  $g(N, \Delta)$  contains only leaves, vertices of degree $\Delta$ and at most one vertex of degree $r $,  $2 \leq r \leq \Delta-1$.

Let $x_j$ denote the number of vertices of degree $j$ in  a tree $T$ having  $N$ edges.  Then:
\begin{itemize}
\item{$\sum_{j=1}^{N+1} x_j=N+1$}
\item{$\sum_{j=1}^{N+1} jx_j=2N$}
\end{itemize}

We consider the following cases:
\smallskip

\noindent \emph{Case 1}:  $x_r  = 0$   for $2 \leq r \leq  \Delta -1$.  Then counting edges we have $x_1 + \Delta x_{\Delta} = 2N$ and counting vertices we have $x_1 +  x_{\Delta}  =  N +1$.

Subtracting  we get  $( \Delta – 1)x_{\Delta} =  N – 1$.     Hence  $ k = 0$,   $x_{\Delta} =  \frac{N-1}{\Delta-1}$  and $g(N,\Delta)=  \frac{(N-1)\binom{\Delta}{2}}{\Delta-1} =  \frac{(N-1)\Delta}{2}$.
\smallskip

\noindent \emph{Case 2}:  $x_r  = 1$ for exactly one value of $r $, $2 \leq r \leq \Delta -1$. Then counting edges we have $x_1 +r + \Delta x_{\Delta} = 2N$ and counting vertices we have $x_1 + 1+ x_{\Delta}  =  N +1$.

Subtracting we get $(r -1 ) + (\Delta-1)x_{\Delta} =  N-1$.   Hence   $k = r-1 \geq 1$,     $x_{\Delta} =  \frac{N- r}{\Delta-1} = \frac{N – k -1}{\Delta-1}$ and \[g(N,\Delta) =  \frac{(N- k-1 )\binom{\Delta}{2}}{\Delta-1}  + \binom{k +1}{2} \leq \frac{\Delta(N-1)}{2},\] proving item 2.}

\item{Suppose $\Delta \geq 2$ and let $T$ be a tree whose line graph realises $g(N,\Delta)$.  Add an edge incident with a leaf of $T$  we obtain a tree $T^*$ such that $\Delta(T^*) = \Delta(T)$, $e(T^*) = e(T) +1  = N +1$ and $e(L(T^*)) = e(L(T) +1$, proving the claim.}
\item{For $\Delta = 1$, the result is trivial.  Suppose $N > \Delta \geq 2$  and let $T$ be a tree whose line graph realises $g(N,\Delta)$.  Since $N > \Delta$, there is a leaf $v$ of $T$ and a vertex $u$ of degree $\Delta$ which are not adjacent.  Drop this leaf $v$ which was adjacent to some vertex $w$ with $deg(w) \leq \Delta$, and add a leaf adjacent to $u$  to obtain a tree $T^*$.  Clearly $e(T^*) = e(T) = N$, $\Delta(T^*) = \Delta(T) +1$ and \[e(L(T^*))  =  e(L(T)) + \left(\binom{\Delta +1}{2} - \binom{\Delta}{2} \right)  - \left(\binom{deg(w)}{2}-  \binom{deg(w) – 1}{2} \right)\]\[= e(L(T)) +  \Delta  - ( deg( w) -1)  \geq  e(L(T))+1,\] proving the claim.}

\end{enumerate}
\end{proof}


\begin{theorem}[Gauss] \label{gausstriangle}
Every non-negative integer is the sum of at most three triangular numbers (hence of exactly three triangular numbers) \cite{ewell}
\end{theorem}

\begin{lemma} \label{nonfstar}
The smallest value of $M $ for which  $( N,M)$  is a non-feasible pair for the line graphs of star-forests  satisfies $M \geq   \binom{N-\lceil (-15+\sqrt{153+72N})/2\rceil}   {2}  > \frac{N^2}{2}  - c_0N\sqrt{N}$  for some positive  constant $c_0 $.   
\end{lemma}
\begin{proof}
 
Given $N$  and $M = \binom{k}{2}$ where $0 \leq k \leq N$, the star forest  $F_k = K_{1,k} \cup  ( N - k) K_2$  has $n(L(F_k)) = N$ vertices,   and $e(L(F_k)) =  \binom{k}{2}$ edges.  So all  triangular numbers in the range $[0, \ldots, \binom{N}{2} ]$ form feasible $(N,M)$ pairs.   

Suppose we try to cover all values of $M$ in $[ \binom{N – t}{2} +1 ,\ldots, \binom{ N – t +1}{2}) -1 ]$, an interval between consecutive triangular numbers.  This interval has length exactly $N -t-1$.  If we realise all its members by line graphs of star-forests having $N$ edges and maximum degree $\Delta$,  then all pairs  $(N,M)$ in this range are feasible, and we therefore cover the whole range between consecutive triangular numbers.

We use Theorem \ref{gausstriangle} above.  So to cover $\binom{ N - t }{2}$  let us use $N -t$ edges for  the line graph of $ K_{1,N-t }$  and we need $t$ more edges to get a line graph with $N$ vertices and $M$ edges in the range to cover.  Suppose $x \leq y\leq z$   and $\binom{x}{2}   + \binom{y}{2} +\binom{z}{2}=  p \leq  \binom{t/3}{2}$.  Then  clearly $z \leq t/3$ otherwise $\binom{z}{2} > \binom{t/3}{2} \geq p$. Hence $x +y+ z =  t - q$ for some $q \geq 0$,   and $p$ can be realised by the line graph of $K_{1,x} \cup  K_{1,y} \cup K_{1,z} \cup   qK_2$.
We can then cover all the interval  from $\binom{N-t}{2}$  up to $\binom{N-t}{2} +\binom{t/3}{2}$ by line graphs of star-forests on $N$ vertices and $M$ edges.

Hence if $\binom{t/3}{2}\geq  N - t - 1$, we cover all the interval up to the next triangular number. Solving we get $\frac{( t/3)(t/3 – 1)}{2 } \geq  N-t -1$, and simplifying and solving the quadratic inequality gives
$  t  \geq  \frac{ -15  + \sqrt{153 +72N}}{ 2 }$ and since $t$ is an integer it suffices that $t \geq \left  \lceil \frac{-15+\sqrt{153+72N}}{2} \right  \rceil$.

Hence  we can cover the interval $ \binom{N-t}{2} ,\ldots , \binom{N-t+1}{2}  - 1$ as long as  $N - t  \leq   \binom{N-\lceil(-15+\sqrt{153+72N})/2\rceil}   {2}$, so that the first non-feasible pair $(N,M)$  is  for some $M \geq \binom{N-\lceil(-15-\sqrt{153+72N})/2\rceil}   {2}    \geq N^2/2   - c_0N\sqrt{N}$ (where $c_0=5$ is a valid choice).
\end{proof}


We now consider the upper bound for the smallest  non-feasible pair $(N,M)$ for line graphs of all acyclic graphs. 

\begin{lemma} \label{smallest}
The smallest $M$ for which  $(N,M)$  is a non-feasible pair for line graphs of acyclic graphs  satisfies $M \leq  \frac{N^2}{2}  - c_1N\sqrt{N}$  for some $c_1>0$.  
\end{lemma}
\begin{proof}
We consider the pair $(N,M) = (N  ,  \binom{N - t +1}{2} -1 )$.

For every acyclic graph $F$ with $e(F) = N$ and  $\Delta(F)  \geq  N - t +1$  we have $e(L(F)) \geq \binom{N - t +1}{2}$,  and hence $(N,M)$  cannot  be realised by $L(F)$.  So we may assume $F$ is an acyclic graph with  $e(F) = N$  and $\Delta(F) = \Delta \leq N-t$.

Assuming $\Delta(F)  \leq \frac{N}{2}$ we get, by Lemma \ref{lemma3.2} part 2, that $e(L(F)) \leq \frac{(N-1)\Delta}{2}  < \frac{(N-1)N}{4}$,  which by Lemma \ref{nonfstar} is  smaller (for $N$ large enough) than the minimum non-feasible pair $(N,M)$ for the line graphs of all star-forest graphs.

Hence we may assume that to realise $M$ (or values greater than  $M$)  we need to consider trees  (since by Lemma \ref{tree} , $g(N,\Delta)$  is realised only by line graphs of trees) having  $N$ edges and  $\Delta >  N/2$.

Hence let $T$ be a tree on $N + 1$ vertices and  $N$ edges and maximum degree $\Delta$, whose line graph $L(T)$ realises  $g(N, \Delta)$,  where $N/2 < \Delta \leq N-t$. 
Let $d_1,d_2, \ldots,d_{N +1}$ be the degree sequence of $T$.  By Lemma  \ref{lemma3.2} part 4 we may assume  that  $d_{N+1} = \Delta = N-t > N/2$,  because, if the line graph $L(T)$ has $e(L(T)) < M$ then for any line graph of a tree  $T^*$ (or acyclic graph)  on $N$ edges  and $\Delta \leq N-t$, we will have $e(L(T^*)) < M$, by the monotonicity of $g(N,\Delta)$ (and thus $(N,M)$ would be a non-feasible pair).

Having $\Delta(T) = N - t$, we have $t$ more  edges to use and convexity (used in the same way as in the proof of Lemma \ref{lemma3.2} part 2) once again forces that the second largest value must satisfies  $d_N = t + 1 \leq d_{N+1} = N - t$.  

The degree sequence of $T$ is now quite simple, as  it contains exactly  $N -1$ terms equal to 1, one term equal $t +1$ and one term equal  $N-t$,  with sum equal to $N -1 +t+ 1 +N -t = 2N$.  This is a degree sequence of a tree by Fact 2, and is realised by the double star $T =  S_{t, N-t-1}$, obtained by two adjacent vertices $u$ and $v$, where $u$ is adjacent to $N - t - 1$ leaves and $v$ is adjacent to $t$ leaves.

Clearly $e(L(T)) =  \binom{N -t}{2}  + \binom{t+1}{2}$ and we seek the value of $t$ for which  $e(L(T))  <  M = \binom{N  - t +1}{2}  - 1$ which  will show that $(N,M)$ is a non-feasible pair.

We compare $\binom{N -t}{2}+ \binom{ t +1}{2}  \leq \binom{N-t +1}{2} -1$.    Rearranging we get  \[t^2 + 3t - 2N + 2 \leq 0.\] 

Solving for $t$  we get
 \[t = \frac { -3  + \sqrt{9 + 8n – 8 }}{2}   = \frac{ -3 +\sqrt{8n + 1}}{2}.\]

In case this expression is an integer it means equality holds above and we have to decrease $t$ by one.  Otherwise this expression suffices.

It follows that the first pair $(N,M)$ which is non-feasible occurs for some $M  \leq  \frac{N^2}{2}  - c_1N\sqrt{N}$  for some positive constant $0 < c_1 \leq c_0$.

\end{proof}
 
\begin{theorem}
The smallest non-feasible pair $(N,M)$  for the family of  line graphs of all acyclic graphs appears for some  $M$ in the range $N^2/2  - c_0N\sqrt{N}    \leq  M  \leq  N^2/2  - c_1N\sqrt{N}$ for some constants $c_0 \geq c_1 >0$.
\end{theorem}

\begin{proof}
The lower bound is proved in Lemma \ref{nonfstar}, while the upperbound is proved in Lemma \ref{smallest}.
\end{proof}

\section{The family of all line graphs}

As we did with acyclic graphs we first give the following definition: \[f(N, \Delta) = \max \{ e(L(G)): e(G)  =  N \mbox{, }\Delta(G) =  \Delta \mbox{, } \delta(G) \geq 1 \}.\]

Recall that isolated vertices in $G$  are not represented in $L(G)$ and have no impact on $e(G)=n(L(G))$, $\Delta(G)$ and $e(L(G))$.   

The function $f(N,\Delta)$  will play a similar role in this section to the role of $g(N, \Delta)$  in Section 3.  However as we seek the exact determination of all non-feasible pairs $(N,M)$ for the family of all lines graphs,  there are more technical details to overcome and we need several preliminary lemmas before we are able to prove Theorem 1 and Theorem 2.

 \begin{lemma} \label{realisef}
Suppose $\Delta \geq 2$  There exists a connected graph $G$, whose line graph $L(G)$ realises $f(N,\Delta)$, and furthermore $n(G)  \leq N +1$.
\end{lemma}
\begin{proof}
Suppose, by contradiction, that no connected graph $G$  has a line graph $L(G)$ realizing $f(N,\Delta)$.  Assume $G$ is a disconnected graph with at least two components  $A$ and $B$ in $G$.   We consider 3 cases:

\medskip

\noindent \textbf{Case A}:  both $A$ and $B$ are trees.

\smallskip

We identify two vertices $x$ and $y$ of degree 1 and from $A$ and $B$.  Replace them by a new vertex  $w$ of degree $2 \leq \Delta$ adjacent to the neighbour of $x$ and the neighbour of $y$ to get $G^*$.  Then $e(G^*)=e(G)=N$, $\Delta(G^*)=\Delta(G)=\Delta$ but $G^*$ has a smaller number of components and $e(L(G^*))> e(L(G))$ by convexity.

\medskip

\noindent \textbf{Case B}:   $A$ contains a cycle while $B$ is a tree with a leaf $y$.

\smallskip

We consider three cases:

\begin{enumerate}[i.]
\item{if in $A$ there is a vertex $x$ of degree  less than $\Delta$, delete $y$ and connect $x$ to $w$ which is the vertex in $B$ adjacent to $y$. The obtained graphs $G^*$ has $\Delta(G^*)=\Delta(G)=\Delta$, $e(G^*)=e(G)$, but $G^*$ has a smaller number of components and $e(L(G^*))> e(L(G))$ by convexity.  So we may assume that all vertices in $A$ are of degree $\Delta$.}
\item{if in $ A$ all  vertices are of degree $\Delta=2$ then clearly all other components of $G$ are either cycles or paths with a total number of edges equal to $N$.  But $f(N,2) =  N$ which is realised  by  a union of  cycles and by the cycle $C_N$ which is connected.}
\item{if  in  $ A$ all  vertices are of degree $\Delta \geq 3$ then there is an edge $e = xy$  in $A$ such that  $A\backslash \{ e\}$  is connected (since $A$ contains a cycle).  We consider the following two possibilities:

\begin{enumerate}[a.]
\item{if $B$ is a single edge then we replace $B$ by a vertex $w$, delete $e = xy$ and connect $x$ and $y$ to $w$  to get $G^*$. Clearly $e(G^*)=e(G)$, $\Delta(G^*) = \Delta(G)=\Delta$, $G^*$ has a smaller number of components and $e(L(G^*)) > e(L(G))$ by convexity}
\item{if $B$ is a tree on at least two edges it must have at least two leaves  $u$ and $v$.  Let $u^*$ and $v^*$ be the vertices in $B$ adjacent to $u$ and $v$ (observe that it is possible that  $u^* = v^*$).  We delete $u$ and $e=xy$, and connect $x$ to $v$ and $y$ to $u^*$ to get $G^*$ where again $e(G^*)=e(G)$, $\Delta(G^*) = \Delta(G)=\Delta$, $G^*$ has a smaller number of components and $e(L(G^*)) > e(L(G))$ by convexity.}
\end{enumerate}}
 \end{enumerate}

\medskip

\noindent \textbf{Case C}:   both $A$ and $B$ are not trees.

\smallskip

In this case then $A$ contains an edge $e=xy$ and $B$ an edge $e^*=uv$  such that deleting $e$ and $e^*$  leaves $A$ and $B$ connected (deleting edges from cycles in $A$ and $B$).  We delete $e$ and $e^*$  and connect $x$ to $u$  and $y$ to $v$ to get $G^*$ with $e(G^*) =e(G) = N$, $\Delta(G^*)=\Delta(G)=\Delta$, $G^*$ has a smaller number of components and $e(L(G^*))=  e(L(G))$.

\medskip

So we can always reduce the number of components preserving the number of edges  $N$ and the maximum degree $\Delta$ without reducing the number of edges in the obtained line graphs, and we can repeat this process until we get a connected graph $G^*$ where $L(G^*)$ realises $f(n,\Delta)$.  

Now we consider any connected graph $G$  whose line graph realises $f(N,\Delta)$.  Clearly  $e(G) = N \geq \Delta$ and by connectivity of $G$,  $n(G)  \leq e(G) +1  = N +1$.  
\end{proof}

\begin{lemma} \label{monof}
The monotonicity of $f(N,\Delta)$:
\begin{enumerate}
\item{$f(N,1) =  0$.}
\item{$f(N+1, \Delta ) \geq f(N,\Delta)+1$ for $N \geq \Delta \geq 2$.}
\item{$f(N,\Delta+1) \geq f(N,\Delta) +1$  for  $N > \binom{\Delta+1 }{2}$.}
\item{$f(N,\Delta+1) \geq f(N,\Delta)$ for $\Delta  > \frac{N}{2}$ and  $f(N,\Delta+1) \geq f(N,\Delta) +1$ for  $\Delta > \lceil \frac{N}{2} \rceil$.}
\end{enumerate}
\end{lemma}
\begin{proof}
$\mbox{ }$\\
\begin{enumerate}
\item{For $\Delta = 1$,  $f(N,\Delta) = 0$,  since the line graph of $NK_2$ is the independent set on $N$ vertices.}
\item{Let $G$ be a connected graph such that $L(G)$ realises $f(N,\Delta)$, guaranteed by Lemma \ref{realisef}.  We add a vertex $w$ and delete an edge $e = xy$ in $G$ and connect $x$ and $y$ to $w$  to obtain $G^*$.  Clearly  $e(G^*)= e(G) +1$ and $\Delta(G) = \Delta(G^*)$  while $e(L(G^*)) = e(L(G)) +1$.}
\item{Consider $v$ of maximum degree $\Delta$ in $G$ where $L(G)$ realises $f(N,\Delta)$. We consider three cases:
\begin{enumerate}[i)]
\item{there is an edge $e=xy$ in $E(G)$  such that both $x$ and $y$ are not adjacent to $v$.  Assume $deg(y) \leq deg(x)\leq deg (v)$.  We replace the edge $xy$ by the edge $xv$ to get $G^*$ with $N$  edges and $\Delta(G^*) = \Delta +1$ and convexity gives $e(L(G^*))  > e(L(G))  = f(N,\Delta)$.}
\item{there is an edge $e=xy$ in $E(G)$ with $x$ not adjacent to $v$.  Again we replace  the edge $xy$  by the edge $xv$ to get $G^*$ with $N$ edges and $\Delta(G^*) = \Delta +1$, and convexity gives $e(L(G^*))  > e(L(G))  = f(N,\Delta)$.}
\item{for every edge $xy$,  $v$ is adjacent to both  $x$ and $y$.  Hence $v$ is adjacent to all other vertices of $G$ and $n(G)= \Delta+1$.  So  $N \leq \binom{\Delta+1}{2}$ a contradiction that $N  >  \binom{\Delta +1}{2}$.}
\end{enumerate}

Hence for $N> \binom{\Delta+1}{2}$ we get that  $f(N,\Delta+1) \geq f(N ,\Delta ) +1$. }
\item{Consider again case (iii) above(for the other two cases above the proof remains the same as in item 3).  So a vertex $v$ of maximum degree $\Delta > N/2$  is adjacent to all other vertices of $G$.  Since $\Delta > N/2$   it follows that in the graph  $H = G \backslash \{ v \}$, $e(H)  = N- \Delta   < N/2$     and $\Delta(H) <  N/2$, and the maximum sum of degrees in $G$  of adjacent vertices  $x$, $y$ in $H$  with $deg(x) \leq deg(y)$  is at most
\[deg(x) +deg(y) =  (deg_H(x) +1) +(deg_H(y) + 1) =  deg_H(x) +deg_H(y)  +2\]\[ \leq  (e(H)+1) +2   =  N -\Delta  +3.\]   

We delete  the edge $xy$ and  attach it as a leaf to $v$ to obtain a graph $G^*$.  Clearly $e(L(G^*))-e(L(G))  \geq $
\begin{align*}
&\binom{\Delta +1}{2}  - \binom{\Delta }{2}-\left( \binom{deg(x)}{2}  - \binom{deg(x)-1}{2}  + \binom{deg(y)}{2}  -  \binom{deg(y)-1}{2}\right)  \\
&=\Delta -  (deg(x)-1)- (deg(y) -1)  =  \Delta +2  - (deg(x) +deg(y)) \\
&\geq \Delta+2  -  ( N- \Delta +3 )=2\Delta-(N+1) \geq 0
\end{align*}

since $ \Delta > N /2  $, with equality possible only  if   $\Delta = (N+1)/ 2$. }
\end{enumerate}
\end{proof}

Here are a few examples demonstrating sharpness in  items 3 and 4
\begin{itemize}
\item{$f(6,3) = 12$, $f(6,4 )  = 11$, $f(6,5) = 12$.}
\item{$f(9,4)=f(9,5)= f(9,6) =  24$ and $f(9,7) = 26$.}
\end{itemize}

\begin{lemma} \label{maxedges}

The maximum number of edges in a line graph $L(G)$  of a graph $G$ with $m$ edges is $\binom{m }{2}$ and it is only realised by $G=  K_{1,m}$, and in case $m = 3$ also by $G = K_3$.
\end{lemma}
\begin{proof}
Suppose $G$ has $m$ edges.  Consider any edge $e$ of $G$.  This edge can be incident with at most  $m-1$ edges (if it is incident with all of them)  hence the degree of $e$ as a vertex in $L(G)$ is at most $m-1$.  Hence $\Delta(L(G)) \leq m-1$ and $e(L(G)) \leq \frac{m(m-1)}{2}$ which is realised by $L(K_{1,m})$ and in case $m=3$, also by $L(K_3)$.   

It is well known that the only connected graphs containing no $2K_2$ (and hence every edge is incident with all edges)  are $K_{1,m}$ and $K_3$.

Therefore for graphs with $m$ edges, $\max(e(L(G)))$ is realised precisely by the line graph of $K_{1,m}$ and in case $m = 3$ also by the line graph of $K_3$.
\end{proof}

\begin{lemma} \label{extremal}

Suppose $N \geq 2t +1$,  $t \geq 1$ and let $G$ be a graph with $\Delta(G)  = N-t > \frac{N}{2}$  and $e(G) = N$, and suppose $L(G)$ realises $f(N,\Delta) = f(N,N-t)$.  Then  $f(N,\Delta) =  \binom{N-t}{2} +\binom{t +2 }{2}  - 1$   and the structure of $G$ is  determined as follows :
\begin{enumerate}
\item{$G$  contains two adjacent vertices $u$  and $v$  such that  $deg(u) = \Delta(G) $ and $deg(v) =  t +1$, and  all the t neighbours of $v$ (except $u$)  are also neighbours of $u$. This graph will be denoted by $ Q(N,t)$.}
\item{if $t  = 3$ then there is a second extremal example  where $deg(v) = 3$,  $u$, $x$ and $y$ are the vertices adjacent to $v$, and $x$ and $y$ are also adjacent in $N(u)$. This graph will be denoted by $Q^*(N,3)$.}
\end{enumerate}

 \end{lemma}

\begin{proof}

Let $G$ be a graph such that $L(G)$ realises $f(N,\Delta)$ with $\Delta=N-t>\frac{N}{2}$.  Clearly, since $e(G) = N$, it follows that $n(L(G)) = N$.  Let $u$ be a vertex of maximum degree $N- t$.  Then $|N(u)| = N-t \geq t +1$.  Let $H$ be any graph on $t$ edges without isolated vertices,  and suppose $E(G)  = E(K_{1,N-t}) \cup  E(H)$. 

We consider two cases:

\smallskip

\noindent Case A:  $H$ is not connected

\smallskip

Then there are at least two connected components $A$ and $B$.  Hence there are vertices $x \in A$ and $y \in B$ with both $deg(x)$ and $deg(y)$ at least 1, and $N_H[x]$ and $N_H[y]$ are disjoint in $G$, otherwise $x$ and $y$ are adjacent to a vertex $z$  by edges in $E(H)$ and they are in the same component.  Now
\begin{enumerate}[i)]
\item{suppose there are vertices $x \in A$ and $y \in B$ both not in $N(u)$,  with $deg(x)$ and $deg(y)$ the degrees in $H$ (and in $G$).   Identify $x$ and $y$ to a vertex $w$, namely replace $x$ and $y$ by a vertex $w$ adjacent to all neighbours of $x$ and $y$, to obtain $G^*$ with $e(G^*)=e(G)$ and $\Delta(G^*)=\Delta(G)$.  Observe that  $deg(w) = deg(x) +deg(y)  <  \Delta(G)$, since all edges incident with $x$ and $y$ are distinct from the edges incident with $u$  and  $deg_{G^*}(w)=deg(x) +deg(y ) \leq e(H)=t <\Delta(G)$.  Clearly since only $x$,$y$ and $w$ are involved in creating $G^*$ and no other vertex changes its degree, we have 
\[ e(L(G^*)) - e(L(G))  =\binom{deg(w)}{2} - \left (\binom{deg(x)}{2}+ \binom{deg(y)}{2} \right )=\]\[\binom{deg(x)+deg(y)}{2} -\left (\binom{deg(x)}{2} + \binom{deg(y)}{2} \right)  = deg(x)deg(y) > 0.  \]}
\item{suppose there are vertices $x \in A  \not \in N(u)$  and  $y \in B$ contained in $N(u)$.  Identify $x$ and $y$ to a vertex $w$, as before, to obtain $G^*$ with $e(G^*)=e(G)$ and $\Delta(G^*)=\Delta(G)$.  Because $deg_{G^*}(w) \leq \Delta(G)$, since $y$ is  adjacent to $u$ in $G$, but all other edges incident with $x$ and $y$ are distinct from the edges incident with $u$. Hence $deg_{G^*}(w) =  deg(x) +deg(y) +1  \leq t +1 \leq \Delta(G)$.   Since $N[x]$ and $N[y]$ are disjoint in $G$,  it is clear that the only degrees changed are those of $x$, $y$ and $w$.  Now 
\[e(L(G^*)) - e(L(G))  =\]\[ \binom{deg(w)}{2} -\left (\binom{deg(x)}{2}  +\binom{deg(y+1)}{2} \right )=  \]\[\binom{deg(x)+deg(y)+1}{2}   -  \left (\binom{deg(x)}{2}   +\binom{deg(y)+1}{2} \right )=  \]\[deg(x)deg(y) +deg(x) > 0.\]}
\item{suppose $V(H)$ is contained in $N(u)$.  We cannot use identification of a vertex $x \in A$ and $y \in B$ into a vertex $w$ since this forces $w$ to have a multiple edge to $u$ and graphs with multiple edges are not allowed.  

Assume without loss of generality that $x$ has maximum degree in $A$, $y$ has maximum degree in $B$ and $deg_H(x) \leq deg_H(y)$.  Let $z$ be a vertex in $A$ adjacent in $A$  to $x$.     Clearly, in $G$, $deg(z)  \leq deg(x)$.   Observe that $x$ and $y$  are not adjacent in $G$ as their only common neighbour is $u$.  

We claim that, in $G$, $deg(y)  < deg(u)  = \Delta(G) = N-t$.  This is because otherwise $deg(y)= deg(u)$ but  this forces $y$ to be adjacent to all $N(u) \backslash \{ y \}$ and in particular to $x$ which is impossible since they belong to distinct components of  $H$.

So we delete the edge $e = zx$  and connect $x$ to $y$  to get $G^*$ with $e(G^*)  = e(G)$   and $\Delta(G^*)  = \Delta(G)$.  The only vertices whose degrees have changed are $z$ and $y$ (and $z$ remains adjacent to $u$).  Now   
\[e(L(G^*)) - e(L(G))  = \binom{deg(y)+1}{2} - \binom{deg(y)}{2}+\left(\binom{deg(z)-1}{2}  -\binom{deg(z)}{2}\right ) \]\[=  deg(y)-deg(z) +1 \geq 1\]
since $deg(y) \geq deg(z)$.  Observe  that the graph $H^*$ obtained from $H$ by deleting $e = zx$ and connecting $x$ to $y$ does not necessarily decrease the number of components  in $H^*$ with respect to $H$, but $e(L(G^*))  > e(L(G))$  contradicting the fact that $L(G)$  realises  $f(N,\Delta)$.}

\end{enumerate} 
It follows that in all the above cases, if $H$ is not connected then $L(G)$  is not a graph realizing $f(N, \Delta)$.
\smallskip

\noindent Case B:  $H$ is connected.

\smallskip

Since $H$ is connected and has t edges it follows that $n(H) \leq t +1$ (realised only if $H$ is a tree) and $\Delta(H) \leq t$.
Now assume there is a vertex $w \in V(H) \backslash N(u)$.  Then since $N \geq 2t+1$ it follows that $|N(u)|  = N-t  \geq t +1$  while $|H| \leq t +1$.  Hence  there must be a vertex $v \in N(u)$ incident only to $u$. Delete $w$ and all edges connecting  all vertices of $H$ to $w$ and instead connect them all to $v$ to obtain another graph $G^*$.   Clearly  $e(G^*) = e(G) = N$  and $\Delta(G^*)  = \Delta(G)$ as the only vertices whose degrees were changed are $w$ which was deleted and $v$  where $deg_G(v) = 1$ and $deg_G^*(v) =1 + deg_H(w) \leq  1 + \Delta(H) \leq  1+ t \leq \Delta(G)$.   

Hence we get \[e(L(G^*)) -e(L(G)) =  \binom{deg(v) + deg(w)}{2}  -   \binom{deg(w)}{2}=\]\[\binom{deg(w)+1}{2}  - \binom{deg(w)}{2}  = deg(w)  > 0.\]


We conclude that to maximize $e(L(G))$ subject to $e(G)=N$ and $\Delta(G)=\Delta$ it must be that  all the $t$ edges of $H$ are packed in $N(u)$,  and since $H$ contains no isolated vertices it follows that  $V(H)$ is fully contained in $N(u)$, for otherwise we can continue to apply the step above. 

So now we have a graph $G$ on  $N-t +1= |N[u]| $ vertices  with maximum degree  $N-t \geq  t +1$, and the rest of the $t$ edges form a graph $H$ packed in $N(u)$.  We claim that to maximize $e(L(G))$ it suffices to maximize $e(L(H))$, which  by Lemma \ref{maxedges} can only be done if $H$ is $K_{1,t}$ and in case $t = 3$ also by $H =  K_3$.

 To prove this claim, suppose $H$ and $F$ are two connected  graphs having $t$ edges that we want to pack in $N(u)$.  Let the degrees of $H$ be $x_1,\ldots,x_r$ and the degrees of $F$ be $y_1,\ldots,y_s$.  Observe that \[\sum_{i=1}^{r} x_i = \sum_{j=1}^{s} y_j = 2t\] and assume $e(L(H))  > e(L(F))$.

We denote by $H^*$ respectively $F^*$ the graphs obtained by packing $H$ respectively $F$ in $N(u)$.  Clearly 
\[e(L(H^*))  - e(L(F^*)) =\]\[  \sum_{i=1}^{r} \binom{ x_i +1}{2} -\sum_{j=1}^{s} \binom{y_j +1}{2}  =\sum_{i=1}^{r} \left (\binom{x_i}{2} + x_i \right ) -  \sum_{j=1}^{s} \left (\binom{y_j}{2} + y_j \right )=\]\[  \sum_{i=1}^{r} \binom{x_i}{2} + \sum_{i=1}^{r} x_i  -  \sum_{j=1}^{s} \binom{y_j}{2} -\sum_{j=1}^{s} y_j  =\sum_{i=1}^{r} \binom{x_i}{2}  -  \sum_{j=1}^{s} \binom{y_j}{2} =\]\[ e(L(H)  - e(L(F)) > 0,\]

 proving the claim.

Hence by Lemma \ref{maxedges}, $H$ is  $K_{1,t}$ and in case $t = 3$, $H$ can be also  $K_3$.

Now the extremal  graph $Q(N,t )$ is uniquely determined by the degree sequence, namely  there are two adjacent vertices $u$ and $v$ with $deg(u)  = \Delta(G) = N-t$  and  $deg(v)  = t +1$  and there are $t$ more vertices of degree 2 while the rest have degree 1.

In case $t = 3$ and  $\Delta = N - 3$ there is another extremal graph, denoted $Q^*(N,3)$,  obtained  by packing $K_3$ in $N(u)$ where  $deg(u) = \Delta$.  
Furthermore a simple calculation reveals that $e(L(Q(N,t)))  =   \binom{N-t}{2}+ \binom{t +2}{2}  - 1  =  f(N,\Delta)$,  and monotonicity of $f(N,\Delta)$ for fixed $N$ and $\Delta > \frac{N}{2}$ follows once again by convexity since  $t < N /2$,  so decreasing $t$ increases $f(N,\Delta)=f(N,N-t)$.          

 \end{proof}

\begin{figure}[H]
\centering
\includegraphics[scale=0.8]{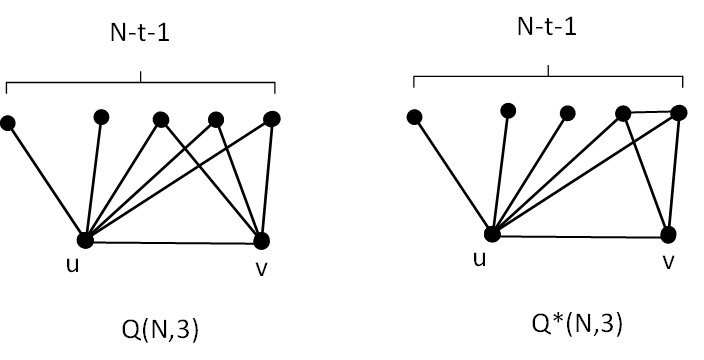} 
\caption{Q graphs}
\label{fig1}
\end{figure}

Since Lemma \ref{extremal} deals with the case where $\Delta > N/2$, the next Lemma deals with the case $\Delta=N/2$.

\begin{lemma} \label{evenN}
Let $N$ be even.  Then,
\begin{enumerate}
\item{for $N \geq 6$, $ f(N,\frac{N}{2}) =  \frac{N^2}{4} +3$.}
\item{for $N \geq 10$, $f(N,\frac{N}{2} +1)  > f(N,\frac{N}{2} )$.}
\end{enumerate}
\end{lemma}
\begin{proof}
$\mbox{ }$\\
\begin{enumerate}
\item{For $N  = 6$,  $\Delta = 3$ we know that $f(6,3)=  12  =  \frac{N^2}{4} +3$.   

So we assume  that $N \geq 8$.  Consider $G$ with $N$ edges such that $L(G)$ realises $f(N,N/2)$.  We can have at most two vertices of degree  $\frac{N}{2}$  since otherwise if we have three vertices of degree $\frac{N}{2}$ then $e(G) \geq \frac{3N}{2}  - 3  > N$  for $N \geq 7$.  Hence by convexity, to optimize $e(L(G))$ there must be two vertices $u$  and $v$  of degree $\frac{N}{2}$.  These two vertices of degree $\frac{N}{2}$ are incident with at least $N-1$ edges and this occurs when they are adjacent. 

\noindent Case A:  $u$ and $v$ are not adjacent. 

Then they are incident with $N$ edges and every vertex of $G  \backslash \{ u,v \}$  is adjacent to either $u$  or $v$ or both and have degree either 1 or 2.  By convexity $e(L(G))$ is maximized when  all vertices of $G  \backslash \{ u,v \}$ are incident to both $u$  and $v$, and have degree 2, and in this case there are $\frac{N}{2}$ such vertices.  Furthermore $e(L(G)) =  2\binom{ N/2}{2} + \frac{N}{2}  =  \frac{N^2}{4}$. 

\noindent Case B:  $u$  and $v$ are adjacent. 

Then they are incident with $N - 1$ edges of which $N - 2$ edges are adjacent to other vertices.  Consider all the vertices incident to either $u$ or $v$ or both.  Without the edge which is not incident with $u$ and $v$,  all these vertices may have degrees either 1 or 2  and  the last edge can be incident to at most two of these  vertices raising their degree to at most 3.   By convexity, the maximum is realised when there are two vertices of degree $\frac{N}{2}$,  two vertices of degree 3 and the rest are vertices of degree 2.  This is realised by a graph $G$ with two vertices $u$ and $v$  of degree $\frac{N}{2}$, and $\frac{N}{2} -1$ other vertices all adjacent to both $u$  and $v$, and of which a pair of vertices $ x$ and $y$ are  adjacent by an edge and have degree 3.

Now $e(L(G)) =  2\binom{N/2}{2}  +6 +\frac{N}{2} -3  =  \frac{N^2}{4} +3$.
}

\item{Observe that $f(6,3) = 12 > f(6,4) = 11$, $f(8,4)  = f(8,5) = 19$ and $f(9,4) = f(9,5) = f(9,6) =  24$.

But for $N \geq 10$ we equate, applying Lemma \ref{extremal}, $f(N, \frac{N}{2} +1)$  to $f(N,\frac{N}{2}) =  \frac{N^2}{4} +3$   to get
 \[2\binom{\frac{N}{2} +1}{2} -1  =  \frac{N}{2} \left(\frac{N}{2}+1 \right)  - 1  =  \frac{N^2}{4} + \frac{N}{2} - 1  >  \frac{N^2}{4} +3\] which holds true for  $N \geq 10$.}
\end{enumerate}
\end{proof}

Remark : Lemmas \ref{monof} part 4, \ref{extremal} and \ref{evenN} show that  in fact $f(N,\Delta)$  is monotone increasing for $N \geq 8$ and $\Delta\geq \frac{N}{2}$ and strictly monotone increasing for $N\geq 10$ and $\Delta \geq \frac{N}{2}$.      

\begin{lemma} \label{upperbound}
For  $\Delta \leq \frac{N}{2}$ and $N \geq 12$,  $f(N,\Delta )  \leq \frac{N^2}{3}$. 
\end{lemma}
 
\begin{proof}
We observe first that for every $N$ and  $\Delta$, an upper bound for $f(N,\Delta )$ is $N(\Delta-1)$, because each edge $e =xy$ in $G$ has as a vertex in $L(G)$, $deg(e )  = deg(x) +deg(y) - 2 \leq 2(\Delta-1)$. Hence $e(L(G)) \leq  \frac{2N(\Delta-1)}{2} = N(\Delta-1)$.  In particular, for $\Delta \leq \frac{N}{3} +1$,   $f(N,\Delta) \leq \frac{ N^2}{3}$.

We now compute an upper bound for $f(N,\Delta)$  when $\frac{N}{3} +2  \leq \Delta \leq \frac{N}{2}$.  This forces $N  \geq 12$.  Suppose $\frac{N}{3} +2 \leq \Delta =  \frac{N}{3} + t \leq \frac{N}{2}$  where $2\leq  t \leq \frac{N}{6}$.  Applying convexity on the ``proposed degree sequence  of  $G$" (because this is an upper bound that is maximized over all the sequence of integers $1\leq d_1 \leq d_2\leq \ldots \leq d_s  =  \Delta$  such that  $\sum_{j=1}^s  d_j = 2N$, some of which might not be a graphical sequence),  we have two vertices of degree $\Delta$  but we cannot have three vertices of degree $\Delta$ as then the number of edges in $G$ is at least $3\Delta -  3  =  3 \left (\frac{N}{3} +t  - 1 \right )  \geq N +3$.
 
So  we may have two vertices of degree $\frac{N}{3}  + t$  and the third greatest degree is at most $\frac{N}{3}  - 2t  +3$  as again  these three vertices force that the number of edges  in $G$ is at least  $2 \left (\frac{N}{3} +t \right)  + \left( \frac{N}{3} -2t +3 \right)- 3  = N$.

So, by convexity, we have two vertices of degree $\frac{N}{3} +t$, one vertex of degree $\frac{N}{3} -2t + 3$  and the rest of the vertices have degree at most 3 (since the three largest vertices are incident with $N$ edges hence no further edge is possible).  Therefore there are most $\frac{N}{3}$  vertices
 of degree 3 since $\sum deg(v) = 2N$. So
\[e(L(G)) \leq  2\binom{\frac{N}{3} +t}{2}  + \binom{\frac{N}{3} -2t +3}{2}  + \frac{3N}{3}.\]

Writing $x = \frac{N}{3}$, then   $t\leq \frac{x}{2}$   we get 
\[2\binom{x  +t }{2} + \binom{x- 2t +3}{2} + N   = \frac{ 2(x+t)(x+t - 1) + (x-2t +3)(x - 2t +2) +6x }{2}=\]\[   \frac { (  2x^2 +2xt -2x +2tx +2t^2 - 2t ) + (  x^2 -2xt +2x -2xt +4t^2  -4t +3x  - 6t +6 ) +6x }{2}=\]
\[\frac{3x^2 +9x +  6t^2 - 12t  +6}{2}  \leq \frac{3x^2 +9x +  6x^2/4  - 12x/2 +6}{2}=  \frac{12x^2 +36x + 6x^2  -  24x +24}{8} =\]\[ \frac{18x^2  +12x  +24}{8}=\frac{9x^2 + 6x + 12}{4 } =  \frac{N^2 + 2N  +12}{4}  < \frac{ N^2}{3}\] already for  $N  \geq 10$. Hence the assumption $N \geq 12$ suffices.

Hence  we get  $f(N,\Delta ) \leq \frac{N^2}{3}$  for  $\Delta\leq  \frac{N}{2}$  and $N \geq 12$. 

\end{proof}


For a given positive integer  $N$ we shall use the phrase that an interval of consecutive integers $[ a,a +1.. b-1,b ]$ is feasible if all pairs $(N,M)$,  $a \leq M \leq b$ are feasible pairs.

\begin{lemma} \label{intfeasdelta}
The interval  $[ \binom{\Delta}{2} ,\ldots, f(N,\Delta ) ]$ is feasible for $\Delta  > \frac{N}{2}$.
\end{lemma}

\begin{proof}
Let $\Delta=N-t$, $t< \frac{N}{2}$.  We begin, at step 0, with the graph $G_0=K_{1,N-t} \cup tK_2$, with the vertex $u$ having degree $N-t$ with neighbours $v$ and $u_1,\ldots, u_{N-t-1}$.  Then $L(G_0)$ has $N$ vertices and $\binom{N-t}{2}$ edges.  We take the following steps:
\begin{enumerate}
\item{We remove one $K_2$ and add a new vertex $x_1$ incident to  $v$ in $K_{1,N-t}$.  This is graph $G_1$ and $L(G_1)$ has $N$ vertices and $\binom{N-t}{2}+1$ edges.  Then identify $x_1$ with $u_1$ and the line graph now has $\binom{N-t}{2}+2$ edges}
\item{We remove  $2K_2$ and add vertices $x_1$ and $x_2$ to $v$ to give the graph $G_2$ such that $L(G_2)$ has $\binom{N-t}{2}+3$ edges.  Again, identifying $x_1$ with $u_1$ gives a line graph with $\binom{N-t}{2}+4$ edges and then identifying $x_2$ with $u_2$ gives a line graph with $\binom{N-t}{2}+5$ }

\item{So at step $k$ we remove $kK_2$  and attach $k$ leaves  $x_1,\ldots,x_k$ to $v$ to give $G_k$ so that $L(G_k)$ has $\binom{N-t}{2}+\binom{k+1}{2}$ edges.  Then, step by step, we identify $x_j$ with $u_j$ for $j=1 \ldots k$, so that at each step we increase the number of edges in the line graph by 1.  We can thus cover the interval $[\binom{N-t}{2}+\binom{k+1}{2}, \ldots,\binom{N-t}{2}+\binom{k+2}{2}-1]$.  Figure \ref{fig2} illustrates this process.}
\end{enumerate}

Since there are  $t$ independent edges initially, we can continue up to step $t$, and the last move in step $t$ will give us a line graphs with $N$ vertices and $\binom{N-t}{2} +\binom{t+2}{2}-1=f(N,\Delta)$ edges (as proved in Lemma \ref{extremal}).

\begin{figure}[H]
\centering
\includegraphics[scale=0.5]{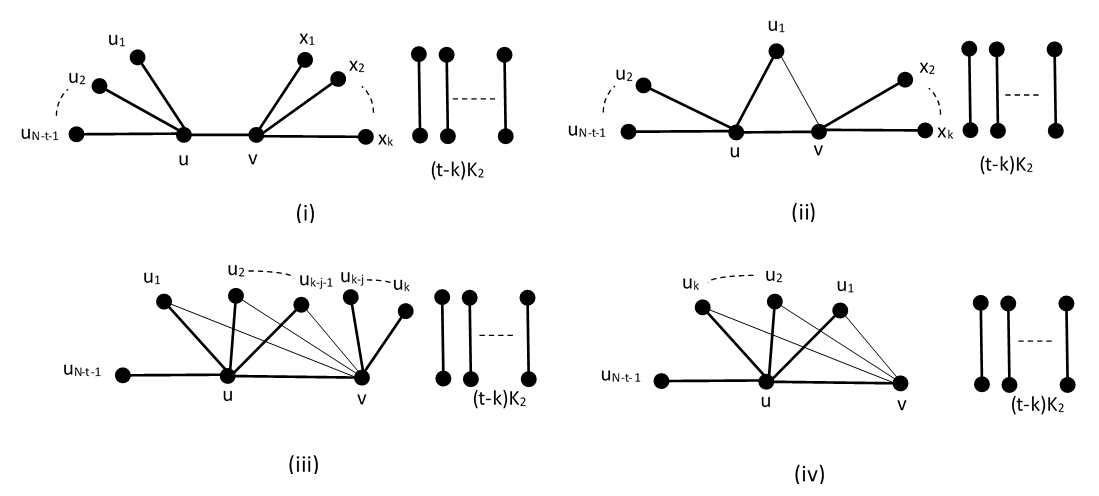} 
\caption{Step k}
\label{fig2}
\end{figure}
\end{proof}

\bigskip

\begin{lemma} \label{intfeas}
The interval $[0, \binom{\lfloor \frac{N}{2} \rfloor +1}{2} - 1 ]$ is feasible.
\end{lemma}
\begin{proof}
We shall write $\Delta = N - t$ and assume $\Delta \leq \frac{ N}{2}$.  We will show that the interval $[0,\ldots, \binom{\lfloor \frac{N}{2} \rfloor +1}{2} - 1 ]$ is feasible by dividing it into subintervals depending on $\Delta$.

We consider the graph $G =  K_{1, N-t}  \cup  P_k  \cup (N-t-k)K_2$ where $P_k$ is the path on $k$ edges and  $ 0 \leq k \leq N-t$.  Clearly  $M= 0$  is realised by the line graph of $NK_2$  corresponding to the case $\Delta = 1$.

So assume  $\Delta \geq 2$  and observe that $\binom{\Delta +1}{2}  - 1 - \binom{\Delta}{2}=  \Delta-1$.  The line graph of   $K_{1, N-t} \cup   P_k  \cup  ( N-t-k)K_2$  has  $N$ vertices and $\binom{N-t }{2} +    (k-1)\binom{2}{2}  = \binom{(N-t }{2}  + (k-1)$  edges and since $1 \leq k \leq N-t$, the feasible pairs range from $M =  \binom{N - t}{2}$ up to $\binom{N-t}{2} +  N-t - 1$.

Hence the interval from $ \binom{\Delta}{2}$ to $\binom{\Delta}{2}+ N-t -1  \geq  \binom{\Delta}{2} + \Delta  - 1   =  \binom{\Delta +1}{2}-1$  is feasible.  This holds for  all  $\Delta \leq \left \lfloor \frac{N}{2} \right \rceil$ and the interval $[0,\ldots, \binom{\lfloor \frac{N}{2} \rfloor +1}{2} - 1 ]$ is feasible.

\end{proof}

We are now ready to prove Theorem 1 as stated in the introduction.
 \begin{theorem1}[The Intervals Theorem] 
For $N  \geq 5$,  all the values of $M$ for which $(N, M)$ is a non-feasible pair for the family of all line graphs, are exactly  given by all integers $M$ belonging to the following intervals:
\[\left [ \binom{N-t}{2} +\binom{t+2}{2}, \ldots,  \binom{N-t+1}{2}-1 \right ] \mbox{ for }  1 \leq t < \frac{ -5 + \sqrt{8N +17}}{2}\]

Observe that if  $\frac{ -5 + \sqrt{8N +17}}{2}$ is not an integer then $t= \left \lfloor \frac{ -5 + \sqrt{8N +17}}{2} \right \rfloor$ while if $\frac{ -5 + \sqrt{8N +17}}{2}$ is an integer then $t=\frac{ -5 + \sqrt{8N +17}}{2}-1$.

\end{theorem1}
 
\begin{proof} 

By  Lemma \ref{intfeas},  the  pairs $(N,M)$, for $M$ in the interval  $[0,\ldots, \binom{\lfloor \frac{N}{2} \rfloor +1}{2} - 1 ]$ are feasible.

By Lemma \ref{intfeasdelta}, for $\Delta = N-t  > \frac{N}{2}$, the interval  $[\binom{N-t}{2} \ldots \binom{N-t}{2}  + \binom{t +2}{2}-1]$ is fully realised and feasible. 

The right hand side of the above interval is the maximum possible value that can be attained by a line graph of a graph with $N$ edges and   $\Delta  = N -t > \frac{N}{2}$  as proved in  Lemma \ref{extremal}.

So we first try to cover, for $\Delta  = N-t > \frac{N}{2}$,  the interval $[\binom{\Delta}{2}  = \binom{N-t}{2},  \binom{\Delta + 1 }{2} -1   =\binom{N-t  +1}{2} -1]$  using graphs with $\Delta  = N-t > \frac{N}{2}$.  

Now the gap between these  triangular numbers can be covered by feasible values if and only if  $f(N,\Delta)  \geq  \binom{N - t +1 }{2}-1$.  This is because $f(N , \Delta)$  is monotone increasing for $\Delta > \frac{N}{2}$, and by Lemma \ref{intfeasdelta},  the interval   $[ \binom{\Delta}{2},f(N,\Delta]$  is feasible.  If however $f(N,\Delta)  <  \binom{N- t  +1}{2} -1$  then  all the interval  $[ f(N,\Delta) +1 \ldots  \binom{N-t +1 }{2} - 1 ]$  is non-feasible unless it can be covered by line graphs of graphs on $N$ edges and $\Delta \leq \frac{N}{2}$ which we will show to be impossible. 

Recall from  Lemma \ref{extremal}  that  for  $\Delta >\frac{N}{2}$, $f(N,\Delta ) =  \binom{N-t}{2} + \binom{t +2}{2} - 1$.  Hence we equate $f(N,\Delta) = f(N,N-t) =  \binom{N-t}{2} + \binom{t +2 }{2}  - 1 <     \binom{N - t  +1 }{2} -1$  and we shall express $t$ as a function of $N$.  Now

\[\frac{(N-t)(N-t-1)}{2} + \frac{(t+2)(t+1)}{2} < \frac{(N-t +1)(N-t)}{2}\]   hence   \[(t+2)(t+1) < (N-t)( N-t +1 -( N-t -1) )  = 2(N-t).\] This gives the quadratic  $t^2 +3t +2 < 2N-2t$  or $t^2 +5t +2 - 2N <0$.

Solving we get  $t =  \frac{-5 +  \sqrt{25  -8 + 8N}}{2}  = \frac{ -5 + \sqrt{8N +17}}{2}$.  Observe that if the expression is an integer than  it corresponds to equality above meaning it covers the end of the interval and we have to take $t -1$,  otherwise we take $\lfloor t \rfloor$.

Lastly  we have to show that  these non-covered values of $M$ by line graphs of graphs with $\Delta >  \frac{N}{2}$  cannot be covered by line graphs of graphs with $\Delta \leq \frac{N}{2}$.

However by Lemma \ref{upperbound}, for $N \geq 12$  and for $\Delta \leq \frac{N}{2}$,  we know that $f(N,\Delta ) \leq \frac{ N^2}{3}$.  So we have to consider $\frac{N^2}{3} \leq \binom{N - \frac{-5  + \sqrt{8N +17}}{2} }{2}$  which holds  for $N \geq 33 $


Now computing the values of $N$  for which new intervals start  we have the exact values stated in the theorem  which is thus proved for $N \geq 33$. 
 
For  $1 \leq N \leq 35$  we wrote a simple computer program  that checks which pairs $(N, M)$ are feasible and which are not.  This is a simple Mathematica program which generates all partitions of $2N$ into graphical sequences $d_1,\ldots,d_n$ and computes  $\sum_{j=1}^{n} \binom{d_j}{2}$.

This program confirmed Theorem 1 in this range, which completes the proof of Theorem 1 as well as the proof of Theorem 2 presented  below.

\end{proof}
\begin{theorem2}[The minimum non-feasible pair]
For  $N \geq 2$,  the minimum  value of $M$ which makes $(N,M)$ a non-feasible pair, for the family of all line graphs,  is $\binom{N-t}{2} + \binom{t +2}{2}$  where :
\begin{enumerate}
\item{$t  = \left \lfloor \frac{ -5 + \sqrt{8N +17}}{2} \right \rfloor$    if  $\frac{ -5 + \sqrt{8N +17}}{2}$  is not an integer.}
\item{$t =   \frac{ -5 + \sqrt{8N +17}}{2}  - 1$ if  $\frac{ -5 + \sqrt{8N +17}}{2}$  is an integer.}
\end{enumerate}
\end{theorem2}
\begin{proof}
We know from Theorem 1  that (for $N \geq 33$)  the equation determining, for given $N$,  the minimum $M$ making $(N,M)$ a non-feasible pair  is given by equating  \[f(N,\Delta) = f(N,N-t) =  \binom{N-t}{2}+ \binom{t +2}{2} - 1 <     \binom{N - t  +1}{2}-1.\]

As we have seen, solving for $t$ gives \[  t =\frac{ -5 + \sqrt{25-8+8N }}{2}=\frac{ -5 + \sqrt{8N +17}}{2}.\]

Recall that $t$ must be an integer,  hence if $t$  is not an integer we take $\lfloor t \rfloor$.  When $t$ is an integer, this is because we have equality, namely $\binom{  N - t +1}{2}  - 1  = f(N, N-t) = f(N,\Delta)$.  So in this case we have to replace $t$ by $t-1$, as explained in Theorem 1 and the theorem is proved for $N \geq 33$.
 
For $2 \leq N \leq 35$  we have computed  the minimum non-feasible $M$ using the above mentioned program which completes the proof. 

\end{proof}

\subsection*{Acknowledgements}

We would like to thank the referees whose careful reading of the paper and feedback have helped us to improve this paper considerably.

\bibliographystyle{plain}
\bibliography{linegraphsfeas1}
 
\end{document}